\documentclass[a4paper,12pt,oneside]{article}
\usepackage{amsmath, amssymb, graphicx, amsthm, mathrsfs}
\usepackage{mathptmx}
\usepackage{mathtools} 
\usepackage{geometry}
\usepackage{enumitem}  % For customizing itemize environment
\usepackage{hyperref}
\usepackage{tikz}
\usepackage{pgfplots}
\pgfplotsset{compat=1.18}
\newtheorem{theorem}{Theorem}[section]
\newtheorem{definition}[theorem]{Definition}
\newtheorem{lemma}[theorem]{Lemma}

\newtheorem{corollary}[theorem]{Corollary}

\newtheorem{remark}[theorem]{Remark}
\newtheorem{example}[theorem]{Example}
\numberwithin{equation}{section}

\title{Common Fixed Points of Cq-Commuting Maps via Generalized Gregus-Type Inequalities}
\author{
Babu G.V.R\textsuperscript{*},  Alemayehu G. Negashs\textsuperscript{\dag}, Meaza F. Bogale \textsuperscript{\ddag}\\
\small \textsuperscript{*}Department of Mathematics, Andhra University, India\\
\small \textsuperscript{*}Email: \texttt{gvr\_babu@hotmail.com} \\
\small \textsuperscript{\dag,\ddag}Department of Mathematics, Hampton University, USA\\
\small \textsuperscript{\dag}Email: \texttt{alemayehu.negash@hamptonu.edu}\\
\small \textsuperscript{\ddag}Email: \texttt{meaza.bogale@hamptonu.edu}
}
\date{}

\begin{document}

\maketitle

\begin{abstract}
We establish the existence of common fixed points for $C_q$-commuting self-mappings satisfying a generalized Gregus-type inequality with quadratic terms in $q$-starshaped subsets of normed linear spaces. Our framework extends classical fixed point theory through:
\begin{enumerate}[label=(\roman*)]
    \item Set-distance constraints $\delta(\cdot, [q, \cdot])$ generalizing norm conditions
    \item Compatibility via $C_q$-commutativity without full affinity requirements
    \item Reciprocal continuity replacing full map continuity.
\end{enumerate} 
Explicit examples (e.g., Example 2.6) demonstrate the non-triviality of these extensions. As applications, we derive invariant approximation theorems for best approximation sets. Our results generalize Nashine's work \cite{Nashine2007} and unify several known fixed point theorems.
\end{abstract}

\textbf{AMS(2000) Mathematics Subject Classification:} 47H10, 54H25.

\section{Introduction}
In 1980, Greg\'{u}s established the following existence theorem of fixed points for nonexpansive mappings in a Banach space.

\begin{theorem}[Greg\'{u}s, 1980]
Let $E$ be a closed convex subset of a Banach space $X$ and $T:X\to X$ a mapping that satisfies
\[
\|Tx-Ty\|\leq a\|x-y\|+b\|Tx-x\|+c\|Ty-y\|
\]
for all $x,y\in E$, $0<a<1$, $b\geq 0$, $c\geq 0$ and $a+b+c=1$. Then $T$ has a unique fixed point in $E$.
\end{theorem}

This result has many generalizations such as in Babu \textit{et al.} \cite{BabuPrasad2006}, \'{C}iri\'{c} \textit{et al.} \cite{CiricUme2003}, Fisher \textit{et al.} \cite{FisherSessa1986}, Murthy \textit{et al.} \cite{MurthyEtAl}, Pathak \textit{et al.} \cite{PathakEtAl1995,PathakFisher1999}.

\begin{definition}[Jungck \cite{Jungck1986}]
Let $E$ be a non-empty subset of a metric space $(X,d)$ and let $A$ and $T$ be two selfmaps on $E$. The pair $(A,T)$ is said to be:
\begin{enumerate}
    \item \textit{compatible} if $\lim\limits_{n\to\infty}d(ATx_{n},TAx_{n})=0$, whenever $\{x_{n}\}$ is a sequence in $X$ such that $\lim\limits_{n\to\infty}Ax_{n}=\lim\limits_{n\to\infty}Tx_{n}=t$, for some $t$ in $X$.
    
    \item \textit{reciprocally continuous} \cite{Pant1998} if $\lim\limits_{n\to\infty}ATx_{n}=At$ and $\lim\limits_{n\to\infty}TAx_{n}=Tt$ whenever $\{x_{n}\}$ is a sequence in $X$ such that $\lim\limits_{n\to\infty}Ax_{n}=\lim\limits_{n\to\infty}Tx_{n}=t$, for some $t$ in $X$.
\end{enumerate}
\end{definition}

Evidently any two continuous functions are reciprocally continuous, but its converse need not be true \cite{Pant1998}.

\begin{theorem}[Tas, Telci and Fisher \cite{TasEtAl1996}]
Let $A$, $B$, $S$ and $T$ be selfmaps of a complete metric space $(X,d)$ satisfying:
\begin{enumerate}
    \item $A(X)\subseteq T(X)$ and $B(X)\subseteq S(X)$;
    \item $[d(Ax,By)]^{2}\leq c_{1}$ max $\{[d(Sx,Ax)]^{2},\ [d(Ty,By)]^{2},\ [d(Sx,Ty)]^{2}\}$
    
    $\quad\quad\quad\quad\quad+c_{2}$ max $\{d(Sx,Ax)\ d(Sx,By),\ d(Ty,Ax)\ d(Ty,By)\}$
    
    $\quad\quad\quad\quad\quad+c_{3}$ $d(Sx,By)$ $d(Ty,Ax)$
    
    for all $x,y\in X$, where $c_{1},c_{2},c_{3}\geq 0$, $c_{1}+2c_{2}<1$, $c_{1}+c_{3}<1$, and
    \item the pairs $(A,S)$ and $(B,T)$ are compatible on $X$.
\end{enumerate}
If one of the mappings $A$, $B$, $S$ and $T$ is continuous then $A$, $B$, $S$ and $T$ have a unique common fixed point in $X$.
\end{theorem}

\begin{definition}
Let $E$ be a non-empty subset of a metric space $(X,d)$ and let $A$ and $S$ be two selfmaps on $E$. The pair $(A,T)$ is said to be \textit{weakly compatible} if $TAx=ATx$ whenever $Ax=Tx$, $x\in X$.
\end{definition}

Obviously, every compatible pair is weakly compatible but its converse need not be true \cite{JungckRhoades1998}.

Babu and Kameswari \cite{BabuKameswari2004} generalized Theorem 1.3 by relaxing the continuity of $A$, $B$, $S$, and $T$ and replacing compatible by weakly compatible in (iii). In fact, Kameswari \cite{Kameswari2008} proved the following theorem.

\begin{theorem}[Kameswari \cite{Kameswari2008}]
Let $A$, $B$, $S$ and $T$ be selfmaps of a complete metric space $(X,d)$ satisfying (i) and (ii) of Theorem 1.3 and
\begin{enumerate}
    \item[(iii)] the pairs $(A,S)$ and $(B,T)$ are compatible on $X$.
\end{enumerate}
If either of $A(X)$ or $B(X)$ or $S(X)$ or $T(X)$ is a complete subspace of $X$, then $A$, $B$, $S$ and $T$ have a unique common fixed point in $X$.
\end{theorem}

\begin{definition}[1.6]
Let $X$ be a normed linear space. A subset $E$ of $X$ is said to be:
\begin{enumerate}
    \item \textit{convex} if $\lambda x+(1-\lambda)y\in E$, whenever $x,y\in E$ and $0\leq\lambda\leq 1$;
    
    \item \textit{$q$-starshaped} if there exists at least one point $q\in E$ such that the line segment $[x,q]$ joining $x$ to $q$ is contained in $E$ for all $x,y\in E$ (that is, $\lambda x+(1-\lambda)q\in E$, whenever $x,y\in E$ and $0<\lambda<1$). In this case, $q$ is called the star-center of $E$.
\end{enumerate}
\end{definition}

Every convex set is starshaped with respect to each of its points, but its converse need not be true \cite{BabuAlemayehu}.
\begin{example}
    
 Example of a $q$-Starshaped Set that is Not Convex\\
Consider $\mathbb{R}^2$ with the Euclidean norm. Define the set:
\[
E = \overline{B}((0,0), 1) \cup \overline{B}((2,0), 1)
\]
where $\overline{B}((a,b), r)$ denotes the closed disk of radius $r$ centered at $(a,b)$. Explicitly:
\[
E = \left\{(x,y) : x^2 + y^2 \leq 1\right\} \cup \left\{(x,y) : (x-2)^2 + y^2 \leq 1\right\}.
\]
The \textbf{star-center} is $q = (1,0)$.

\paragraph{Verification of Properties}

\paragraph{(1) $q$-Starshapedness}
For any $p \in E$, the line segment $[q, p]$ lies in $E$:
\begin{itemize}
    \item \textbf{Case 1:} $p \in \overline{B}((0,0), 1)$ \\
    Parametrize $\gamma(t) = (1-t)(1,0) + t(x_p, y_p) = (1 + t(x_p-1), ty_p)$ for $t \in [0,1]$. Then:
    \[
    \|\gamma(t) - (0,0)\|^2 = [1 + t(x_p-1)]^2 + (ty_p)^2 \leq 1
    \]
    since $x_p^2 + y_p^2 \leq 1$ and the function is convex.

    \item \textbf{Case 2:} $p \in \overline{B}((2,0), 1)$ \\
    Parametrize $\gamma(t) = (1 + t(x_p-1), ty_p)$ for $t \in [0,1]$. Then:
    \[
    \|\gamma(t) - (2,0)\|^2 = [t(x_p-1) -1]^2 + (ty_p)^2 \leq 1
    \]
    since $(x_p-2)^2 + y_p^2 \leq 1$ and the function is convex.
\end{itemize}
Thus $E$ is $q$-starshaped with $q = (1,0)$.

\paragraph{(2) Non-Convexity}
Take $a = (0,1) \in E$ and $b = (2,1) \in E$. Their midpoint is:
\[
c = \frac{a + b}{2} = (1,1)
\]
But $(1,1) \notin E$ because:
\[
\|(1,1) - (0,0)\| = \sqrt{2} > 1 \quad \text{and} \quad \|(1,1) - (2,0)\| = \sqrt{2} > 1
\]

\paragraph{Geometric Illustration}
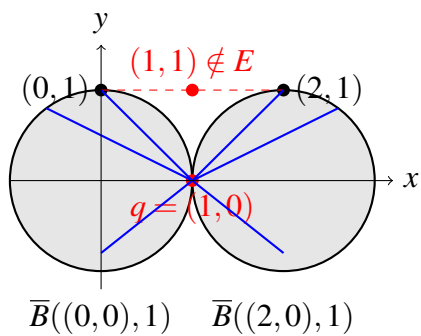
\begin{center}
\begin{tikzpicture}[scale=1.2]
    % Draw disks
    \fill[gray!20] (0,0) circle (1);
    \fill[gray!20] (2,0) circle (1);
    \draw[thick] (0,0) circle (1);
    \draw[thick] (2,0) circle (1);
    
    % Star-center q
    \fill[red] (1,0) circle (2pt) node[below] {$q = (1,0)$};
    
    % Points a and b
    \fill (0,1) circle (2pt) node[left] {$(0,1)$};
    \fill (2,1) circle (2pt) node[right] {$(2,1)$};
    
    % Midpoint c (not in E)
    \fill[red] (1,1) circle (2pt) node[above] {$(1,1) \notin E$};
    
    % Segment between a and b (fails convexity)
    \draw[dashed, red] (0,1) -- (2,1);
    
    % Example segments from q (demonstrate starshapedness)
    \draw[blue, thick] (1,0) -- (0,1);
    \draw[blue, thick] (1,0) -- (2,1);
    \draw[blue, thick] (1,0) -- (-0.6,0.8);
    \draw[blue, thick] (1,0) -- (2.6,0.8);
    \draw[blue, thick] (1,0) -- (0,-0.8);
    \draw[blue, thick] (1,0) -- (2,-0.8);
    
    % Axes
    \draw[->] (-1,0) -- (3.2,0) node[right] {$x$};
    \draw[->] (0,-1.2) -- (0,1.5) node[above] {$y$};
    
    % Labels
    \node at (0,-1.5) {$\overline{B}((0,0),1)$};
    \node at (2,-1.5) {$\overline{B}((2,0),1)$};
\end{tikzpicture}
\end{center}
\end{example}

\paragraph{Key Observations}
\begin{itemize}
    \item \textbf{$q$-Starshapedness}: All segments from $q$ to points in $E$ lie entirely in $E$
    \item \textbf{Non-Convexity}: The segment between $(0,1)$ and $(2,1)$ contains points outside $E$
    %\item \textbf{Applications}: Such sets appear in optimization problems with disconnected feasible regions
\end{itemize}

\begin{definition}[1.7]
Let $E$ be a convex subset of normed linear space $X$. A selfmap $T:E\to E$ is said to be:
\begin{enumerate}
    \item \textit{affine} if
    \[
    T(\lambda x+(1-\lambda)y)=\lambda Tx+(1-\lambda)Ty
    \]
    for all $x,y\in E$ and $\lambda\in(0,1)$;
    
    \item \textit{affine with respect to a point $q$} \cite{VijayarajuMaridai2004} if
    \[
    T(\lambda x+(1-\lambda)q)=\lambda Tx+(1-\lambda)Tq
    \]
    for all $x\in E$ and $\lambda\in(0,1)$;
    
    \item \textit{$q$-affine} if $E$ is $q$-starshaped and
    \[
    T(\lambda x+(1-\lambda)q)=\lambda Tx+(1-\lambda)q
    \]
    for all $x\in E$ and $\lambda\in(0,1)$.
\end{enumerate}
\end{definition}

Here we observe that if $T$ is $q$-affine then $Tq=q$.

\begin{remark}[1.8]
\begin{enumerate}
    \item Every affine mapping is affine with respect to a point but its converse need not be true \cite{VijayarajuMaridai2004,Nashine2007};
    
    \item Every affine map $T$ is $q$-affine if $Tq=q$ but its converse need not be true even if $Tq=q$ \cite{BabuAlemayehu};
    
    \item Every $q$-affine map $T$ is affine with respect to a point $q$, but its converse need not be true as shown by the following example.
\end{enumerate}
\end{remark}

\begin{example}[1.9]
Let $X=\mathbb{R}$ with the usual norm and let $E=[0,1]$. We define a selfmap $T$ on $E$ by
\[
Tx=\begin{cases}
1 & \text{if } x\in[0,1) \\
0 & \text{if } x=1.
\end{cases}
\]
Then we have
\[
T\left(\lambda x+(1-\lambda)\frac{1}{2}\right)=\begin{cases}
1 & \text{if } x\in[0,1) \text{ and } \lambda\in(0,1) \\
0 & \text{if } x=1 \text{ and } \lambda=1.
\end{cases}
\]

If $x\in[0,1)$ and $\lambda\in(0,1)$, then $T(x)=1=T\left(\frac{1}{2}\right)$ and therefore
\[
T\left(\lambda x+(1-\lambda)\frac{1}{2}\right)=1=\lambda Tx+(1-\lambda)T\left(\frac{1}{2}\right).
\]

If $x=1$ and $\lambda=1$, then
\[
T\left(\lambda x+(1-\lambda)\frac{1}{2}\right)=0=\lambda Tx+(1-\lambda)T\left(\frac{1}{2}\right).
\]

Therefore, $T$ is affine with respect to $\frac{1}{2}$.

But $T\left(\frac{1}{2}\right)=1$, so $T$ is not $q$-affine for $q=1/2$ because $Tq \neq q$.

%So, $T$ is not $\frac{1}{2}$-affine.
\end{example}

\begin{definition}[1.10, \cite{Al-Thagafi2006}]
Let $E$ be $q$-starshaped subset of a normed linear space $X$. Let $A,T:E\to E$. A pair $(A,T)$ is called $C_{q}$-commuting if $ATx=TAx$ for all $x\in C_{q}(A,T)$, where $C_{q}(A,T):=\bigcup\{C(A,T_{k}):k\in[0,1]\}$ where $C(A,T_{k})=\{x\in E:Ax=T_{k}x\}$ and $T_{k}x=kTx+(1-k)q$.
\end{definition}

We note that each pair of $C_{q}$-commuting mappings is weakly compatible but its converse need not be true \cite{BabuAlemayehu}.

%\section*{Example of $C_q$-Commuting Maps}

\begin{example}
    
Consider the normed linear space $X = \mathbb{R}$ with the absolute-value norm and the $q$-starshaped set $E = [0, 1]$ with \textbf{star-center} $\boldsymbol{q = 0}$. Define selfmaps $A, T: E \to E$ as:
\[
A(x) = 
\begin{cases} 
0 & \text{if } 0 \leq x \leq 0.5, \\
1 & \text{if } 0.5 < x \leq 1,
\end{cases}
\quad \text{and} \quad
T(x) = x^2.
\]

\paragraph{Step 1:}
For $k \in [0, 1]$, define 
\(
T_k(x) = k T(x) + (1-k)q = kx^2.
\)
The coincidence set $C(A, T_k) = \{x \in E : A(x) = T_k(x)\}$ is:
\begin{itemize}
    \item \textbf{Case $k = 0$}: $T_0(x) = 0$. Then $A(x) = 0$ $\forall x \in [0, 0.5]$, so $C(A, T_0) = [0, 0.5]$.
    \item \textbf{Case $k \in (0, 1)$}: 
    \begin{itemize}
        \item For $x \in [0, 0.5]$, $A(x) = 0$ but $T_k(x) = kx^2 > 0$ when $x > 0$, so only $x=0$ satisfies $A(0) = 0 = T_k(0)$.
        \item For $x \in (0.5, 1]$, $A(x) = 1$ but $T_k(x) = kx^2 < 1$ (since $k < 1$ and $x^2 \leq 1$), so no solutions.
        Thus $C(A, T_k) = \{0\}$.
    \end{itemize}
    \item \textbf{Case $k = 1$}: $T_1(x) = x^2$. Solve $A(x) = x^2$:
    \begin{itemize}
        \item If $x \in [0, 0.5]$, $0 = x^2 \implies x = 0$.
        \item If $x \in (0.5, 1]$, $1 = x^2 \implies x = 1$.
        Thus $C(A, T_1) = \{0, 1\}$.
    \end{itemize}
\end{itemize}
The $C_q$-set is:
\[
C_q(A, T) = \bigcup_{k \in [0, 1]} C(A, T_k) = [0, 0.5] \cup \{1\}.
\]

\paragraph{Step 2:} 
Check $ATx = TAx$ for all $x \in C_q(A, T)$:
\begin{itemize}
    \item \textbf{Case $x \in [0, 0.5]$}: 
    $Tx = x^2 \leq 0.25 \leq 0.5$, so $A(Tx) = 0$. \\
    $A(x) = 0$, so $T(Ax) = T(0) = 0$. Thus $ATx = TAx = 0$.
    
    \item \textbf{Case $x = 1$}: 
    $T(1) = 1$, so $A(T1) = A(1) = 1$. \\
    $A(1) = 1$, so $T(A1) = T(1) = 1$. Thus $AT1 = TA1 = 1$.
\end{itemize}
Hence, $(A, T)$ is $C_q$-commuting.

\paragraph{Step 3:} Checking non-Commuting Outside $C_q(A, T)$.
Pick $x = 0.6 \notin C_q(A, T)$:
\begin{itemize}
    \item $T(0.6) = 0.6^2 = 0.36$, so $A(Tx) = A(0.36) = 0$ (since $0.36 \leq 0.5$).
    \item $A(0.6) = 1$ (since $0.6 > 0.5$), so $T(Ax) = T(1) = 1$.
    \item Thus $AT(0.6) = 0 \neq 1 = TA(0.6)$.
\end{itemize}
\end{example}

\begin{definition}[1.11]
Let $X$ be a Banach space. A selfmap $T:E\subseteq X\to X$ is said to be \textit{demiclosed} if every sequence $\{x_{n}\}$ in $E$ such that $x_{n}$ converges weakly to $x\in E$ and $Tx_{n}$ converges strongly to $y\in E$, then $y=Tx$.
\end{definition}

If $\{x_{n}\}$ converges weakly to $x$ as $n\rightarrow\infty$, we write it by $x=w-\lim_{n\rightarrow\infty}x_{n}$.

\begin{definition}[1.12]
Let $X$ be a Banach space. A selfmap $T:E\to E$ is said to be \textit{weakly continuous} if $Tx_{n}\to Tx$ weakly whenever $x_{n}\to x$ weakly.
\end{definition}

\begin{definition}[1.13]
Let $X$ be a normed linear space and let $E$ be a non-empty subset of $X$. Let $u\in X$. An element $y\in E$ is called a \textit{best approximant} to $u\in X$ if
\[
\|u-y\|=\delta(u,E)=\inf\{\|u-z\|:z\in E\}.
\]
\end{definition}

We write the set $P_{E}(u)=\{x\in E:d(x,y)=\delta(u,E)\}$ is called the \textit{set of best approximants} to $u\in X$ out of $E$.

The aim of this paper is to prove the existence of common fixed points for $C_{q}$-commuting maps satisfying 'Gregus type condition' in a $q$-starshaped domain. As an application, we give invariant approximation results. Our theorems generalize the results of Nashine \cite{Nashine2007}.

\section{Main Results}

Now we prove our main results.

\begin{theorem} \label{Thrm: 2.1}
Let $E$ be a nonempty $q$-starshaped subset of a normed linear space $X$. Let $A$, $B$, $S$ and $T$ be four selfmaps of $E$ satisfying:
\begin{enumerate}
    \item $A(X)\subseteq T(X)$ and $B(X)\subseteq S(X)$; % Range conditions
    
    \item The generalized Gregus-type inequality:
    
    \begin{align}
    \|Ax-By\|^{2} &\leq c_{1} \max \{\left[\delta(Sx,[q,Ax])\right]^{2},\left[\delta(Ty,[q,By])\right]^{2},\|Sx-Ty\|^{2}\} \nonumber\\
    &+c_{2}\max\{\delta(Sx,[q,Ax]) \delta(Sx,[q,By]),\delta(Ty,[q,By]) \delta(Ty,[q,Ax])\} \label{(2.1.1)}\\
    &+c_{3} \delta(Sx,[q,By]) \delta(Ty,[q,Ax]),\nonumber
    \end{align}
    for all $x,y\in X$, where $c_{1},c_{2},c_{3}\geq 0$, $c_{1}+2c_{2}=1$, $c_{1}+c_{3}=1$;
    
    \item Either the pairs $(A,S)$ and $(B,T)$ or $(A,T)$ and $(B,S)$ are reciprocally continuous on $E$; and
    
    \item Both pairs $(A,S)$ and $(B,T)$ are $C_{q}$-commuting on $X$ and $S$ and $T$ are $q$-affine.
\end{enumerate}
If either $S(E)$ or $T(E)$ is compact, then $F(A)\cap F(B)\cap F(S)\cap F(T)\neq\emptyset$.
\end{theorem}

\begin{proof}
Choose a sequence $\{k_{n}\}\subset(0,1)$ with $\lim\limits_{n\to\infty}k_{n}=1$. Define for each $n\geq 1$ and for all $x\in E$, mappings $A_{n}$ and $B_{n}$ by
\[
A_{n}x=k_{n}Ax+(1-k_{n})q \quad \text{and} \quad B_{n}x=k_{n}Bx+(1-k_{n})q.
\]

Since $E$ is $q$-starshaped, $S$ and $T$ are $q$-affine, $A(E)\subseteq T(E)$ and $B(E)\subseteq S(E)$, we get $A_{n}x\in T(E)$ and $B_{n}x\in S(E)$.

Hence, $A_{n}(E)\subseteq T(E)$ and $B_{n}(E)\subseteq S(E)$, $n=1,2,3,\ldots$.

Also, for all $x,y\in E$,

\begin{align} 
\|A_{n}x-B_{n}y\|^{2} &= k_{n}^{2}\|Ax-By\|^{2} \nonumber\\
&\leq k_{n}^{2} c_{1} \max \{\left[\delta(Sx,[q,Ax])\right]^{2},\left[\delta(Ty,[q,By])\right]^{2},\|Sx-Ty\|^{2}\} \nonumber\\
&\quad +k_{n}^{2} c_{2}\max\{\delta(Sx,[q,Ax]) \delta(Sx,[q,By]),\delta(Ty,[q,By]) \delta(Ty,[q,Ax])\} \label{(2.1.2)}\\
&\quad +k_{n}^{2} c_{3} \delta(Sx,[q,By]) \delta(Ty,[q,Ax])\nonumber
\end{align}

Letting $a_{n}=k_{n}^{2} c_{1}$, $b_{n}=k_{n}^{2} c_{2}$ and $c_{n}=k_{n}^{2} c_{3}$, we have $a_{n},b_{n},c_{n}\geq 0$ and $a_{n}+2b_{n}<1$ and $a_{n}+c_{n}<1$.

Also, $\delta(Sx,[q,Ax])=\inf\{\|Sx-y\|:y\in[q,Ax]\}\leq\|Sx-y\|$ for all $y\in[q,Ax]$ for each $n=1,2,\ldots$.

In particular, for each $n\geq 1$ and for all $x\in E$ we have
\[
\delta(Sx,[q,Ax])\leq\|Sx-A_{n}x\|.
\]

Similarly, for each $n\geq 1$ and for all $x,y\in E$, we have
\[
\delta(Sx,[q,By])\leq\|Sx-B_{n}y\|,\quad \delta(Ty,[q,Ax])\leq\|Ty-A_{n}x\|\quad \text{and}\quad \delta(Ty,[q,By])\leq\|Ty-B_{n}y\|.
\]

Hence, from \eqref{(2.1.2)}, for each $n\geq 1$ and for all $x,y\in E$, we have
\[
\begin{aligned}
\|A_{n}x-B_{n}y\|^{2} &\leq a_{n}\,\max\big\{\|Sx-A_{n}x\|^{2},\|Ty-B_{n}y\|^{2},\|Sx-Ty\|^{2}\big\} \\
&\quad +b_{n}\max\{\|Sx-A_{n}x\|\,\|Sx-B_{n}y\|,\|Ty-B_{n}y\|\,\|Ty-A_{n}x\|\} \\
&\quad +c_{n}\,\|Sx-B_{n}y\|\,\|Ty-A_{n}x\|,
\end{aligned}
\]
where $a_{n},b_{n},c_{n}\geq 0$ and $a_{n}+2b_{n}<1$ and $a_{n}+c_{n}<1$.

Hence, for each $n\geq 0$, maps $A_{n}$, $B_{n}$, $S$ and $T$ satisfy the inequality (ii) of Theorem 1.3 \cite{TasEtAl1996}.

Since the pairs $(A,S)$ and $(B,S)$ are $C_{q}$-commuting and $S$ and $T$ are $q$-affine, if $A_{n}x=Sx$ and $B_{n}(x)=Tx$ for some $x\in E$, then we get $A_{n}Sx=SA_{n}x$ and $B_{n}Tx=TB_{n}x$ for each $n\geq 1$.

Thus, for each $n\geq 1$, the pairs $(A_{n},S)$ and $(B_{n},T)$ are weakly compatible maps on $E$. As either $S(E)$ or $T(E)$ is compact, then either $S(E)$ or $T(E)$ is complete. Therefore, for each $n\geq 1$, maps $A_{n}$, $B_{n}$, $S$ and $T$ satisfy all the conditions of Theorem 1.5 \cite{Kameswari2008} and hence for each $n\geq 1$, there exists a unique $x_{n}\in E$ such that

\begin{equation} \label{(2.1.3)}
A_{n}x_{n}=B_{n}x_{n}=Sx_{n}=Tx_{n}=x_{n},\quad n\geq 1.
\end{equation}
First let us assume $S(E)$ is compact. Since $\{Sx_{n}\}_{n=1}^{\infty}\subset S(E)$, there exists a subsequence $\{Sx_{n_{j}}\}$ of $\{Sx_{n}\}$ such that $\lim_{j\to\infty}Sx_{n_{j}}=z$ (say) in $S(E)$.

Hence, from \eqref{(2.1.3)}, we have
\begin{equation} \label{(2.1.4)}
    \lim_{j\to\infty}x_{n_{j}}=z.
\end{equation}

Therefore, from \eqref{(2.1.3)} using \eqref{(2.1.4)} we get

\begin{equation} \label{(2.1.5)}
\lim_{j\to\infty}A_{n_{j}}x_{n_{j}}=\lim_{j\to\infty}B_{n_{j}}x_{n_{j}}=\lim_{j\to\infty}Sx_{n_{j}}=\lim_{j\to\infty}Tx_{n_{j}}=x_{n_{j}}=z.
\end{equation}

Now assume that $(A,S)$ and $(B,T)$ are reciprocally continuous pairs. From (2.1.4) we get
\[
z=\lim_{j\to\infty}A_{n_{j}}x_{n_{j}}=\lim_{j\to\infty}[k_{n_{j}}Ax_{n_{j}}+(1-k_{n_{j}})q]=\lim_{j\to\infty}Ax_{n_{j}}.
\]

Hence,
\begin{equation} \label{(2.1.6)}
    \lim_{j\to\infty}Ax_{n_{j}}=\lim_{j\to\infty}Sx_{n_{j}}=z.
\end{equation}

As the pair $(A,S)$ is reciprocally continuous, from \eqref{(2.1.6)}, it follows that
\begin{equation} \label{(2.1.7)}
    \lim_{j\to\infty}ASx_{n_{j}}=Az \text{ and } \lim_{j\to\infty}SAx_{n_{j}}=Sz
\end{equation}

Now from \eqref{(2.1.5)} and \eqref{(2.1.7)}, we obtain
\[
\begin{aligned}
z &= \lim_{j\to\infty}A_{n_{j}}Sx_{n_{j}} 
= \lim_{j\to\infty}[k_{n_{j}}ASx_{n_{j}}+(1-k_{n_{j}})q] 
= \lim_{j\to\infty}ASx_{n_{j}} 
= Az,
\end{aligned}
\]
and
\begin{equation}
    z = \lim_{j\to\infty}SA_{n_{j}}x_{n_{j}} 
= \lim_{j\to\infty}[k_{n_{j}}SAx_{n_{j}}+(1-k_{n_{j}})q] 
= \lim_{j\to\infty}SAx_{n_{j}} 
= Sz. \label{(2.1.8)}
\end{equation}

Hence, $Az=Sz=z$.

From \eqref{(2.1.4)} we get
\[
\begin{aligned}
z = \lim_{j\to\infty}B_{n_{j}}x_{n_{j}} = \lim_{j\to\infty}[k_{n_{j}}Bx_{n_{j}}+(1-k_{n_{j}})q] = \lim_{j\to\infty}Bx_{n_{j}}.
\end{aligned}
\]

Hence,
\begin{equation}
    \lim_{j\to\infty}Bx_{n_{j}}=\lim_{j\to\infty}Tx_{n_{j}}=z. \label{(2.1.9)}
\end{equation}

As the pair $(B,T)$ is reciprocally continuous, from \eqref{(2.1.9)}, it follows that
\begin{equation}
    \lim_{j\to\infty}BTx_{n_{j}}=Bz \text{ and } \lim_{j\to\infty}TBx_{n_{j}}=Tz \label{(2.1.10)}
\end{equation}

Now from \eqref{(2.1.5)} and \eqref{(2.1.10)}, we obtain
\[
\begin{aligned}
z = \lim_{j\to\infty}B_{n_{j}}Tx_{n_{j}} = \lim_{j\to\infty}[k_{n_{j}}BTx_{n_{j}}+(1-k_{n_{j}})q] = \lim_{j\to\infty}BTx_{n_{j}} = Bz,
\end{aligned}
\]
and

\begin{align}
z = \lim_{j\to\infty}TB_{n_{j}}x_{n_{j}} = \lim_{j\to\infty}[k_{n_{j}}TBx_{n_{j}}+(1-k_{n_{j}})q] = \lim_{j\to\infty}TBx_{n_{j}} = Tz. 
\end{align} \label{(2.1.11)}

Hence, from \eqref{(2.1.8)} and \eqref{(2.1.11)}, we obtain
\begin{equation}
    Az=Bz=Sz=Tz=z. \label{(2.1.12)}
\end{equation}

The proof is similar when $S(E)$ is compact and the pairs $(A,T)$ and $(B,S)$ are reciprocally continuous on $E$.

In similar way, we can show that $A$, $B$, $S$ and $T$ have a common fixed point when $T(E)$ is compact and either the pairs $((A,S)$ and $(B,T))$ or the pairs $((A,T)$ and $(B,S))$ are reciprocally continuous on $E$.
\end{proof}

\begin{corollary} \label{Cor: 2.2}
Let $E$ be a nonempty $q$-starshaped subset of a normed linear space $X$. Let $A$, $B$, $S$ and $T$ be four selfmaps of $E$ satisfying $(i)$, $(iii)$, $(iv)$ of Theorem \eqref{Thrm: 2.1} and
\end{corollary}

\begin{align}
    \|Ax-By\|\leq \max\{\delta(Sx,[q,Ax]),\delta(Ty,[q,By]),\|Sx-Ty\|\}, \label{(2.2.1)}
\end{align}

for all $x,y\in X$. If either of $S(E)$ or $T(E)$ is compact, then $F(A)\cap F(B)\cap F(S)\cap F(T)\neq\emptyset$.

\begin{proof}
Choosing either $c_{2}=0$ or $c_{3}=0$, we get $c_{1}=1$ in condition (ii) of Theorem \eqref{Thrm: 2.1}, so that \eqref{(2.2.1)} holds and the conclusion follows by Theorem \eqref{Thrm: 2.1}.
\end{proof}

\begin{theorem} \label{Thrm: 2.3}
Let $E$ be a nonempty $q$-starshaped subset of a Banach space $X$. Let $A$, $B$, $S$ and $T$ be four selfmaps of $E$ satisfying (i), (ii) and (iv) of Theorem \eqref{Thrm: 2.1}. Further assume that $S$ and $T$ are weakly continuous on $E$ and $S-A$ and $T-B$ are demiclosed at 0. If either $S(E)$ or $T(E)$ is weakly compact, then $F(A)\cap F(B)\cap F(S)\cap F(T)\neq\emptyset$.
\end{theorem}

\begin{proof}
Choose a sequence $\{k_{n}\}\subset(0,1)$ with $\lim\limits_{n\to\infty}k_{n}=1$. Define for each $n\geq 1$ and for all $x\in E$, mappings $A_{n}$ and $B_{n}$ by
\[
A_{n}x=k_{n}Ax+(1-k_{n})q \text{ and } B_{n}x=k_{n}Bx+(1-k_{n})q.
\]

Since every weakly compact subspace of a Banach space $X$ is complete and since either $S(E)$ or $T(E)$ is compact, then either $S(E)$ or $T(E)$ is complete. Therefore, from the proof of Theorem \eqref{Thrm: 2.1}, for each $n\geq 1$, maps $A_{n}$, $B_{n}$, $S$ and $T$ satisfy all the conditions of Theorem 1.5 \cite{Kameswari2008} and hence for each $n\geq 1$, there exists a unique $x_{n}\in E$ such that
\begin{align}
    A_{n}x_{n}=B_{n}x_{n}=Sx_{n}=Tx_{n}=x_{n}, \quad n\geq 1. \label{(2.3.1)}
\end{align}

First let us assume $S(E)$ is weakly compact. Since $\{Sx_{n}\}_{n=1}^{\infty}\subset S(E)$, there exists a subsequence $\{Sx_{n_{j}}\}$ of $\{Sx_{n}\}$ such that $w-\lim_{j\to\infty}Sx_{n_{j}}=z$ (say) in $S(E)$.\\
Hence, from \eqref{(2.3.1)}, we have
\begin{align}
    w-\lim_{j\to\infty}x_{n_{j}}=z. \label{(2.3.2)}
\end{align}
Therefore, from \eqref{(2.3.1)} using \eqref{(2.3.2)}, we get

\begin{align}
    w-\lim_{j\to\infty}A_{n_{j}}x_{n_{j}}=w-\lim_{j\to\infty}B_{n_{j}}x_{n_{j}}=w-\lim_{j\to\infty}Sx_{n_{j}}=w-\lim_{j\to\infty}Tx_{n_{j}}=z \label{(2.3.3)}
\end{align}
Now from \eqref{(2.3.3)} using \eqref{(2.3.2)} and the weak continuity of $S$ and $T$, we have
\begin{align}
    Sz=Tz=z \label{(2.3.4)}
\end{align}
Now since $\|x_{n_{j}}-Ax_{n_{j}}\|=(1-k_{n_{j}})\|q-Ax_{n_{j}}\|\to 0$ as $n\to\infty$ and $Sx_{n_{j}}=x_{n_{j}}$ for all $j\geq 1$, we have
\begin{align}
    \|Sx_{n_{j}}-Ax_{n_{j}}\|\to 0 \text{ as } n\to\infty. \label{(2.3.5)}
\end{align}
Since $x_{n_{j}}\to z$ weakly as $n\to\infty$, by \eqref{(2.3.5)} and the demiclosedness of $S-A$, we obtain
\begin{align}
    Az=Sz \label{(2.3.6)}
\end{align}
Similarly, since $x_{n_{j}}\to z$ weakly as $n\to\infty$, by the demiclosedness of $T-B$, we obtain
\begin{align}
    Bz=Tz. \label{(2.3.6)}
\end{align}
Hence, from \eqref{(2.3.3)}, \eqref{(2.3.5)} and \eqref{(2.3.6)}, we get
\begin{align}
    Az=Bz=Sz=Tz=z. \label{(2.3.7)}
\end{align}

In a similar way, we can show that $A$, $B$, $S$ and $T$ have a common fixed point when $T(E)$ is weakly compact.
\end{proof}
\begin{corollary} \label{Cor: 2.4}
Let $E$ be a nonempty $q$-starshaped subset of a Banach space $X$. Let $A$, $B$, $S$ and $T$ be four selfmaps of $E$ satisfying (i) and (iii) of Theorem \eqref{Thrm: 2.1}, and
\begin{align}
    \|Ax-By\|\leq \max\{\delta(Sx,[q,Ax]),\delta(Ty,[q,By]),\|Sx-Ty\|\}, \label{(2.4.1)}
\end{align}
for all $x,y\in X$. Further assume that $S$ and $T$ are weakly continuous on $E$ and $S-A$ and $T-B$ are demiclosed at 0. If either of $S(E)$ or $T(E)$ is weakly compact, then
\[
F(A)\cap F(B)\cap F(S)\cap F(T)\neq\emptyset.
\]
\end{corollary}

\begin{proof}
Choosing either $c_{2}=0$ or $c_{3}=0$, we get $c_{1}=1$ in condition (ii) of Theorem \eqref{Thrm: 2.1}, so that \eqref{(2.4.1)} holds and the conclusion follows by Theorem \eqref{Thrm: 2.3}.
\end{proof}

\begin{corollary}[ \cite{Nashine2007}] \label{Cor: 2.5}
Let $S$ and $T$ be selfmaps of a weakly compact subset $E$ of a Banach space $X$ and $\mathcal{J}$ be a family of selfmaps $A:E\to E$ such that $A(E)\subseteq S(E)\cap T(E)$. $A$ commutes with $S$ and $T$. Suppose $E$ is $q$-starshaped with $q\in F(S)\cap F(T)$; $S$ and $T$ are affine with respect to $q$ and continuous in the weak topology. If $S$, $T$ and $A,B\in\mathcal{J}$ satisfy
\begin{align}
    \|Ax-By\|\leq\|Sx-Ty\|, \label{(2.5.1)}
\end{align}
for all $x\neq y\in X$, then
\[
F(A)\cap F(B)\cap F(S)\cap F(T)\neq\emptyset \text{ provided that } S-A \text{ and } T-B \text{ are demiclosed}.
\]
\end{corollary}

\begin{proof}
Follows from Corollary \eqref{Cor: 2.4}.
\end{proof}

\begin{example} \label{Ex: 2.6}
Let $X=\mathbb{R}$ with the usual metric and $E=[0,1]$. We define mappings $A$, $B$, $S$ and $T$ on $E$ by

\begin{align*}
A(x) &= B(x) = 
\begin{cases} 
\frac{2}{3} & \text{if } 0 \leq x < \frac{2}{3} \\
\frac{4}{3} - x & \text{if } \frac{2}{3} \leq x \leq 1 
\end{cases} \\
S(x) &= 
\begin{cases} 
\frac{2}{3} & \text{if } 0 \leq x < \frac{2}{3} \\
1 - \frac{x}{2} & \text{if } \frac{2}{3} \leq x \leq 1 
\end{cases} \\
T(x) &= 
\begin{cases} 
1 - \frac{x}{2} & \text{if } 0 \leq x < \frac{2}{3} \\
\frac{2}{3} & \text{if } \frac{2}{3} \leq x \leq 1 
\end{cases}
\end{align*}

with star-center $q = \frac{2}{3}$.
\end{example}
\paragraph{Verification of Conditions}

\begin{enumerate}[label=(\roman*)]
\item \textbf{$C_q$-commuting pairs}:

\begin{itemize}
\item For $(A,S)$: $C_q(A,S) = E$ since for all $x \in E$:
\begin{align*}
A(x) &= S_k(x) = kS(x) + (1-k)\frac{2}{3} \\
\text{has solution } k &= 
\begin{cases}
\text{any } k \in [0,1] & \text{for } x \in [0,\frac{2}{3}) \\
\frac{1}{2} & \text{for } x \in [\frac{2}{3},1]
\end{cases}
\end{align*}
and $ASx = SAx$ holds everywhere.

\item For $(B,T)$: $C_q(B,T) = \{\frac{2}{3}\}$ since:
\begin{align*}
B(\tfrac{2}{3}) &= T_k(\tfrac{2}{3}) = \tfrac{2}{3} \text{ for all } k \\
B(x) &= T_k(x) \text{ has no solution for } x \neq \tfrac{2}{3} \text{ when } k > 0
\end{align*}
and $BT(\frac{2}{3}) = TB(\frac{2}{3}) = \frac{2}{3}$.
\end{itemize}

\item \textbf{Inequality (2.1.1)}:
With $c_1=1$, $c_2=c_3=0$, verify:
\[
\|Ax - By\| \leq \max\{\delta(Sx,[q,Ax]), \delta(Ty,[q,By]), \|Sx-Ty\|\}
\]
All cases satisfy this, e.g.:
\begin{itemize}
\item When $x,y \in [0,\frac{2}{3})$: Both sides $= 0$
\item When $x \in [\frac{2}{3},1]$, $y \in [0,\frac{2}{3})$:
\[
\|Ax-By\| = \|\tfrac{4}{3}-x - \tfrac{2}{3}\| \leq \tfrac{2}{3} - x \leq \|Sx-Ty\|
\]
\end{itemize}

\item \textbf{Reciprocal continuity}:
All maps are piecewise linear with matching values at $x=\frac{2}{3}$, ensuring:
\[
\lim_{x_n \to \tfrac{2}{3}} ASx_n = A(\tfrac{2}{3}) = \tfrac{2}{3} = SA(\tfrac{2}{3})
\]

\item \textbf{$q$-affineness}:
$S$ and $T$ satisfy:
\[
S(\lambda x + (1-\lambda)\tfrac{2}{3}) = \lambda Sx + (1-\lambda)\tfrac{2}{3}
\]
and similarly for $T$.
\end{enumerate}

Hence, Example \eqref{Ex: 2.6} shows that Theorem \eqref{Thrm: 2.3} is a generalization of Theorem 3.1 of Nashine \cite{Nashine2007}.

\section{Invariant Approximation Results}

In this section, we prove the existence of common fixed points in the set of best approximations by using Theorem \eqref{Thrm: 2.1} and Theorem \eqref{Thrm: 2.3}.\\
In what follows we need the following lemma which we state without proof.

\begin{lemma}[ \cite{HicksHumphries1989}] \label{Lem: 3.1}
Let $E$ be a subset of $X$. Then for any $u\in X$, $P_{E}(u)\subseteq\partial E$ (the boundary of $E$).
\end{lemma}

\begin{theorem} \label{Thrm: 3.2}
Let $E$ be a subset of a normed linear space $X$ and $A,B,S,T:E\to E$ be four selfmaps such that $u\in F(A)\cap F(B)\cap F(S)\cap F(T)$ for some $u\in X$ with $S(\partial E)\subset E$ and $T(\partial E)\subset E$. Assume that $P_{E}(u)$ is $q$-starshaped and compact with $q\in F(S)\cap F(T)$, $S$ and $T$ are affine on $P_{E}(u)$ and $S(P_{E}(u))=P_{E}(u)=T(P_{E}(u))$. Suppose also that selfmaps $A$, $B$, $S$ and $T$ satisfy for all $x,y\in P_{E}(u)\cup\{u\}$,
\begin{equation}
    \|Ax-By\|^{2}\leq
\begin{cases}
\|Sx-Tu\|^{2} & \text{if } y=u \\
\|Su-Ty\|^{2} & \text{if } x=u \\
\begin{aligned}
&c_{1} \max \{[\delta(Sx,[q,Ax])]^{2},[\delta(Ty,[q,By])]^{2},\|Sx-Ty\|^{2}\} \\
&+ c_{2} \max \{\delta(Sx,[q,Ax]) \delta(Sx,[q,By]),\delta(Ty,[q,By]) \delta(Ty,[q,Ax])\} \\
&+ c_{3} \delta(Sx,[q,By]) \delta(Ty,[q,Ax])
\end{aligned} & \text{if } x,y\in P_{E}(u),
\end{cases} \label{(3.2.1)}
\end{equation}
where $c_{1},c_{2},c_{3}\geq 0$, $c_{1}+2c_{2}=1$, $c_{1}+c_{3}=1$; and either of the pairs $(A,S)$ and $(B,T)$ or $(A,T)$ and $(B,S)$ are reciprocally continuous on $P_{E}(u)$. If both the pairs $(A,S)$ and $(B,T)$ are $C_{q}$-commuting on $P_{E}(u)$, then $P_{E}(u)\cap F(A)\cap F(B)\cap F(S)\cap F(T)\neq\emptyset$.
\end{theorem}

\begin{proof}
Let $y\in P_{E}(u)$. Then by Lemma \eqref{Lem: 3.1}, $y\in\partial E$. Since $A(\partial E)\subset E$ and $B(\partial E)\subset E$, we have $Ay,By\in E$. Also, since $Sy,Ty\in P_{E}(u)$, $u\in F(A)\cap F(B)\cap F(S)\cap F(T)$ and $A$, $B$, $S$ and $T$ satisfy the inequality \eqref{(3.2.1)}, we have $\|Ay-u\|\leq\delta(u,E)$ and $\|u-By\|\leq\delta(u,E)$.
\end{proof}
Hence, $Ay,By\in P_{E}(u)$ and hence $A(P_{E}(u))\subseteq S(P_{E}(u))$ and $B(P_{E}(u))\subseteq T(P_{E}(u))$.\\
Therefore, by Theorem \eqref{Thrm: 2.1}, there exists $z\in P_{E}(u)$ such that $Az=Bz=Sz=Tz=z$.\\
Hence, $P_{E}(u)\cap F(A)\cap F(B)\cap F(S)\cap F(T)\neq\emptyset$.

\begin{theorem} \label{Thrm: 3.3}
Let $E$ be a subset of a Banach space $X$ and $A,B,S,T:E\to E$ be four selfmaps such that $u\in F(A)\cap F(B)\cap F(S)\cap F(T)$ for some $u\in X$ with $S(\partial E)\subset E$ and $T(\partial E)\subset E$. Assume that $P_{E}(u)$ is $q$-starshaped and weakly compact with $q\in F(S)\cap F(T)$, $S$ and $T$ are affine on $P_{E}(u)$ and $S(P_{E}(u))=P_{E}(u)=T(P_{E}(u))$. Suppose also that selfmaps $A$, $B$, $S$ and $T$ satisfy the inequality \eqref{(3.2.1)} and $S$ and $T$ are weakly continuous on $P_{E}(u)$; and $S-A$ and $T-B$ are demiclosed at $\theta$. If both pairs $(A,S)$ and $(B,T)$ are $C_{q}$-commuting on $P_{E}(u)$ , then
\[
P_{E}(u)\cap F(A)\cap F(B)\cap F(S)\cap F(T)\neq\emptyset.
\]
\end{theorem}

\begin{proof}
Runs on the same lines as that of the proof of Theorem \eqref{Thrm: 3.2}, where we use Theorem \eqref{Thrm: 2.3} instead of Theorem \eqref{Thrm: 2.1}.
\end{proof}

\begin{remark}
Theorem \eqref{Thrm: 3.3} generalizes Theorem \eqref{Thrm: 3.2} of Nashine \cite{Nashine2007}. Hence, related results due to Brosowski \cite{Brosowski1969}, Hicks and Humphries \cite{HicksHumphries1989}, Jungck and Sessa \cite{JungckSessa1995}, Sahab, Khan and Sessa \cite{SahabKhanSessa1988} and Singh \cite{Singh1979} are generalized and improved.
\end{remark}

\section*{Conclusion}
In this paper, we have established common fixed point theorems for $C_q$-commuting self-mappings satisfying a generalized Gregus-type inequality involving quadratic terms in $q$-starshaped domains (Theorems 2.1 and 2.3). Key advancements include:
\begin{enumerate}[label=(\roman*)]
    \item The use of set-distance constraints $\delta(\cdot, [q, \cdot])$ to generalize classical norm-based inequalities,
    \item Relaxation of affinity requirements via $q$-affineness and $C_q$-commutativity,
    \item Integration of reciprocal continuity and weak continuity in Banach spaces.
\end{enumerate}
Applications to invariant approximation theory (Theorems 3.2 and 3.3) demonstrate the utility of our results in best approximation problems. Example 2.6 explicitly verifies that our theorems strictly generalize Nashine's work \cite{Nashine2007}.

\end{document}